\documentclass[12pt]{amsart}
\usepackage{amssymb}
\usepackage{amscd} 
\usepackage{tikz-cd}

\usepackage{mathtools, mathrsfs, color} 
\usepackage{graphicx}
\usepackage{tikz,tikz-cd}
\usepackage[all,cmtip]{xy}
\numberwithin{equation}{section}

\def\today{\number\day\space\ifcase\month\or   January\or February\or
   March\or April\or May\or June\or   July\or August\or September\or
   October\or November\or December\fi\   \number\year}

\theoremstyle{definition}
\newtheorem{thm}{Theorem}[section]

\newtheorem{prp}[thm]{Proposition}
\newtheorem{dfn}[thm]{Definition}
\newtheorem{cor}[thm]{Corollary}

\newtheorem{rmk}[thm]{Remark}

\newtheorem{exa}[thm]{Example}

\newtheorem{qst}[thm]{Question}

\newcommand{\beq}{\begin{equation}}
\newcommand{\eeq}{\end{equation}}
\newcommand{\beqr}{\begin{eqnarray*}}
\newcommand{\eeqr}{\end{eqnarray*}}
\newcommand{\bal}{\begin{align*}}
\newcommand{\eal}{\end{align*}}
\newcommand{\bei}{\begin{itemize}}
\newcommand{\eei}{\end{itemize}}

\newcommand{\te}{\theta}
\newcommand{\ld}{\lambda}

\newcommand{\ph}{\varphi}
\newcommand{\ps}{\psi}

\newcommand{\Q}{{\mathbb{Q}}}
\newcommand{\Z}{{\mathbb{Z}}}
\newcommand{\R}{{\mathbb{R}}}
\newcommand{\C}{{\mathbb{C}}}

\newcommand{\T}{{\mathbb{T}}}
\newcommand{\N}{{\mathbb{Z}}_{\ge 0}}

\newcommand{\Ne}{{\mathbb{Z}}_{\ge 0}}

\newcommand{\deta}{{\operatorname{det}}}
\newcommand{\Id}{{\operatorname{Id}}}

\newcommand{\Tr}{{\mathrm{Tr}}}
\newcommand{\id}{{\mathrm{id}}}

\newcommand{\Aut}{{\mathrm{Aut}}}

\newcommand{\ca}{C*-algebra}
\newcommand{\nt}{noncommutative tori}

\newcommand{\cp}{crossed product}

\renewcommand{\S}{\subset}

\newcommand{\Cst}{\textup{C}^\ast}

\newcommand{\K}{\mathrm{K}} 
\newcommand{\KK}{\mathrm{KK}} 

\title{Metaplectic transformations and finite group actions on noncommutative tori }   

\author[S.~Chakraborty]{Sayan Chakraborty} 
\address{ Stat-Math unit, Indian Statistical Institute,
203 Barrackpore Trunk Road,
Kolkata 700 108, India}
\email{sayan.c@isical.ac.in}
\author[F.~Luef]{Franz Luef}
\address{Department of Mathematical Sciences, NTNU Trondheim, 7041 Trondheim, Norway}
\email{franz.luef@math.ntnu.no}


\begin{document}

\begin{abstract}In this article we describe extensions of some $\K$-theory classes of Heisenberg modules over higher-dimensional noncommutative tori to projective modules over crossed products of non\-commutative tori by finite cyclic groups, aka noncommu\-tative orbi\-folds. The two dimensional case was treated by Echterhoff, L\"uck, Phillips and Walters. Our approach is based on the theory of metaplectic transformations of the representation theory of the Heisenberg group. We also describe the generators of the K-groups of the crossed products of flip actions by $\Z_2$ on 3-dimensional noncommu\-tative tori.
\end{abstract}

\maketitle \pagestyle{myheadings} \markboth{S.~Chakraborty and F. Luef}{Metaplectic transformations and noncommutative tori}
\hrule

\section{Introduction}\label{Sec:Dfn}

The $n$-dimensional noncommutative torus, $A_{\te}$, is defined as the universal \ca\  generated by 
unitaries $U_1,\dots,U_n$ subject to the relations
\[
U_k U_j =  e^{2 \pi i \te_{jk} }\, U_j U_k\quad\text{for}~~j, k = 1, \cdots, n,
\]
where $\te=(\te_{ij})$ is a skew-symmetric real $n \times n$ matrix. For the 2-dimensional noncommutative torus, since ${\te}$ is determined by only one real number, $\te_{12}$, we will denote $\te_{12}$ by ${\te}$ again and the corresponding 2-dimensional noncommutative torus by $A_{\te}$.

Noncommutative tori are central objects in noncommutative geometry. Rieffel (\cite{ri88}) constructed projective modules (which are known as Heisenberg modules) over all noncommutative tori. These constitute the framework for studying the geometry of noncommutative tori such as connections, curvature and Dirac operators on noncommutative tori. While the 2-dimensional noncommutative torus is quite well understood, there remain quite a large number of open questions for higher dimensional noncommutative tori.

  Crossed product C*-algebras associated to finite group actions on noncommutative tori go back to the work of Bratteli, Elliott, Evans and  Kishimoto. They considered (\cite{BEEK91}) the flip action of $\Z_2$ on two dimensional noncommutative tori and the associated crossed products. Recall that the flip action of $\Z_2$  on an 2-dimensional noncommutative torus is given by mapping the generators $U_i$ to $U^{-1}_i$ for $i=1,2$. Kumjian (\cite{kumjian90}) computed the $\K$-theory of the crossed product $A_{\te} \rtimes \Z_2$ for two dimensional noncommutative tori $A_{\te}.$ (Also see \cite{walters:k-theoryflip} for more detailed analysis on $A_{\te} \rtimes \Z_2.$) For irrational $\te$s, Walters in \cite{walters:k-theoryflip} stated generators of the $\K$-theory of $A_{\te} \rtimes \Z_2$ by showing that the generators of the $\K_0$ group of $A_{\te}$ can be made ``flip invariant".  Later in \cite{Wal00} and \cite{wa04}, Walters considered $Z_4$ and $\Z_6$ actions on two dimensional noncommutative tori. Recall that the following defines $\Z_4$ and $\Z_6$ actions on two dimensional noncommutative tori:

  $$ U_1 \rightarrow U_2, \hspace{.1cm}  U_2 \rightarrow U^{-1}_1 \hspace{.5cm} (\text{for}\hspace{.1cm} \Z_4),$$ 
    $$ U_1 \rightarrow U_2, \hspace{.1cm} U_2 \rightarrow e^{-\pi i \te}U^{-1}_1U_2 \hspace{.5cm} (\text{for}\hspace{.1cm} \Z_6).$$
    For these actions Walters showed that the generators (as projective modules) of $K_0$ of $A_{\te}$ are $\Z_2$ and $\Z_4$ equivariant for irrational $\te$ to construct projective modules over $A_{\te} \rtimes Z_4$ and $A_{\te} \rtimes Z_6$, which constitute generators of the corresponding $K_0$ groups.

    Later, in \cite{ELPW}, Echterhoff, L\"uck, Phillips and Walters studied 2-dimen\-sional $A_\theta$ acted on by a finite cyclic subgroup $F$ of $SL_2(\Z)$. Note that $SL_2(\Z)$ has a canonical action on $\Z^2$, which can be lifted  to $A_\theta$. The previous actions of $\Z_2$, $\Z_4$ and $\Z_6$ are implemented by matrices in $SL_2(\Z)$.  It is demonstrated that the standard canonical projective module over $A_\te$, aka Bott class (which is a completion of $\mathscr{S}(\R)$, Schwartz functions on $\R$), can be made equivariant by the action of $F$ yielding a projective module over the crossed product algebra  $A_\theta \rtimes F$. (A result of Green and Julg (see Proposition \ref{mainprop}) shows that equivariant $\K$-theory elements provide elements in the crossed product.) It is also known that the Bott class along with the identity element generate the $\K_0$ group of the noncommutative torus (see the proof of Lemma 4.8 in \cite{ELPW}).

    Recently,  actions of finite groups on higher-dimensional noncommutative tori have been considered in the article \cite{jh}. Let $W \in GL_n(\Z)$ be the generator of the finite cyclic group $F$ acting on $\Z^n$ with $W^T\theta W = \theta$. Then the authors in \cite{jh} showed that there exists an action of $F$ on $n$-dimensional $A_\theta$  (Section 4). Let us assume that $n$ is an even number, $n = 2m$. To analyse projective modules over the corresponding crossed product algebras, we restrict our analysis to the class of Heisenberg modules which are an appropriate completion of $\mathscr{S}(\R^m),$ denoted by $\mathcal{E}$.  We show that this class of projective modules (which may be thought of as higher dimensional versions of the Bott class) over higher dimensional noncommutative tori can be made $F$-equivariant. The metaplectic action (which we introduce in Section 5) is the key tool in our arguments.  This generalises the previous results of Walters and \cite{ELPW}.
Our main theorem states:

\begin{thm} \label{main:intro}

Let $W$ be the generator of the finite cyclic group $F$ acting on $\Z^n$ with $W^T\theta W = \theta$ and hence on $A_\theta$. Then the metaplectic action of $W$ on $\mathscr{S}(\R^m)$ extends to an action on $\mathcal{E}$ such that $\mathcal{E}$ becomes an $F$-equivariantly finitely generated projective $A_\theta$ module and thus a finitely generated  projective module over $A_\theta \rtimes F$.
\end{thm}

Coming back to the flip case, note that this action can be defined for general $n$-dimensional tori $A_\theta$. In this case any Heisenberg module over $A_\te$ can be extended to a module over the crossed product  $A_\theta \rtimes \Z_2$ (see Section 7). Though, in \cite{FW}, the authors have computed the $\K$-theory of $A_\theta \rtimes \Z_2$ for higher dimensional $A_\theta$, but some computations in \cite{FW} are not clear to us (see Remark \ref{FWremark}). They used an exact sequence by Natsume (\cite{Na}) to compute the $\K$-theory of  $A_\theta \rtimes \Z_2$ (see Section 7) as $A_\theta \rtimes \Z_2$ can be written as $A_{\theta'} \rtimes (\mathbb{Z}_2 * \mathbb{Z}_2 ),$ where $A_{\theta'}$ is an $(n-1)$-dimensional noncommutative torus.  For  crossed products like $A \rtimes (\mathbb{Z}_2 * \mathbb{Z}_2 )$, Natsume's exact sequence looks like 

{\tiny \[
\begin{CD}
\K_0 (A) @>{}>>\K_0 (A\rtimes  \Z_2)\oplus \K_0 (A\rtimes \Z_2)@>{}>> \K_0 (A\rtimes \Z_2*\Z_2)   \\
@A{}AA & &  @VV{e_1}V            \\
 \K_1 (A\rtimes \Z_2*\Z_2) @<<{}<  \K_1 (A\rtimes \Z_2)\oplus \K_1 (A\rtimes \Z_2) @<<{}< \K_1 (A).
\end{CD}
\]} 

In the final section we study this exact sequence especially, the connecting map $e_1$, and relate it to the classical Pimsner--Voiculescu exact sequence.  Recall that for crossed products like $A \rtimes \Z$, the Pimsner--Voiculescu sequence looks like 
{ \[
\begin{CD}
\K_0 (A) @>{}>>\K_0 (A)@>{}>> \K_0 (A \rtimes \Z)   \\
@A{}AA & &  @VV{e_2}V            \\
 \K_1 (A \rtimes \Z) @<<{}<  \K_1 (A) @<<{}< \K_1 (A).
\end{CD}
\]} 
The main result of the final section can be stated as follows:
\begin{thm}\label{main1:intro}
For unital $A$, the connecting maps of the above two sequences commute in the following sense:
\[
\xymatrixrowsep{5pc}
\xymatrixcolsep{5pc}
\xymatrix
{
\K_0(A\rtimes \Z_2*\Z_2) \ar[r]^{e_1} \ar[rd]^{p} &
\K_1 (A)  \\
&\K_0(A \rtimes \Z),\ar[u]^{e_2}}
\]
where $p$ is the map induced by the natural map from $A\rtimes \Z_2*\Z_2 \cong (A \rtimes \Z )  \rtimes \Z_2$ to $M_2(A \rtimes \Z)$.
\end{thm}
 
Using this result, for totally irrational $\theta$ (for definition see \ref{subsec:basis_of_Atheta}), we discuss the $\K$-theory of crossed products of 3-dimensional $A_\theta$ with respect to the flip action and describe the generators of $\K$-theory (see Corollary \ref{3toriflipped}). This explains the computations of \cite{FW} for the three dimensional case (for totally irrational $\theta$). Presumably this can be done as well for the $n$-dimensional case, which we plan to discuss in another paper.

Notation: $e(x)$ will always denote the number $e^{2\pi i x}$; the standard symplectic matrix on $\mathbb{R}^{2m}$ is defined by $J=\left( \begin{array}{cc}
0 &  I_m\\
-I_m  &  
0
\end{array} \right),$ where $I_m$ is the $m\times m$ unit matrix, and $\mathscr{S}(\R^m)$ will denote the space of smooth functions of rapid decay on $\R^m$. 
\\~\\
 Acknowledgments: This research was partially supported by the DFG through SFB 878. The first named author wants to thank Siegfried Echterhoff and Nikolay Ivankov for valuable discussions. 

\section{Basics on twisted group algebras and noncommutative tori}

Let $G$ be a discrete group.  A map $\omega: G\times G \to \mathbb{T}$ is called a \emph{2-cocycle} if 
$$ \omega(x, y) \omega(xy, z)= \omega(x, yz) \omega(y, z)
\label{cocycle}
$$
whenever $x, y, z \in G $, and if 
$$ \omega(x, 1) = 1 = \omega(1, x)\label{units} $$
for any $x \in G$.

The $\omega$-twisted left regular representation of the group $G$ is given by the formula:

$$(L_{\omega}(x)f)(y) = \omega(x,x^{-1}y)f(x^{-1}y)$$ for $f \in l^2(G)$. The reduced twisted group C*-algebra $C^*(G,\omega)$  is defined as the sub-C*-algebra of $ B(l^2(G))$ generated by the $\omega$-twisted left regular representation of the group $G$. More details can be found in \cite{zeller-meier68, ELPW}.

\begin{exa}

Let $G$ be the group $\Z^n$. For each $n \times n$ real skew-symmetric matrix $\theta$, we can  construct a 2-cocycle on this group by defining $\omega_\theta(x, y) = e(\langle -\theta x,y \rangle )$. The corresponding twisted group C*-algebra $C^*(G, \omega_\theta)$ is isomorphic to the $n$-dimensional noncommutative torus $A_\te$, which was defined in the introduction. 

  \end{exa}
\begin{exa}
Suppose $W$ be an $n \times n$ matrix of finite order with integer entries. Let $F := \langle W \rangle$ act on $\Z^n$ by group automorphism and $\theta$ is an $n \times n$ real skew-symmetric matrix. We assume in addition  that $W$ is a $\theta$-symplectic matrix, i.e. $W^T\theta W = \theta$. Then we can define a 2-cocycle $\omega_\theta '$ on $G:=\Z^n \rtimes F$ by $\omega_\theta '((x,s),(y,t)) = \omega_\theta (x,s\cdot y)$.  Sometimes one calls the corresponding group C*-algebra, $C^*(G, \omega_\theta')$, a {\it noncommutative orbifold}. 
 \end{exa}
 \section{Projective modules over noncommutative tori}\label{sec:proj_mod_nt}

We fix  $n=2p + q$ for  $p,\, q \in \N$. Let us choose $\theta
  := \left( \begin{array}{ccc}
\theta_{11} & \theta_{12}\\
  \theta_{21}  & \theta_{22}\
  \end{array} \right)$ any $n \times n$ skew-symmetric matrix partitioned into four sub-matrices $\theta_{11},\theta_{12},\theta_{21},\theta_{22}$ and  $\theta_{11}$ is a $2p\times 2p$ matrix.  We recall the approach of Rieffel \cite{ri88} to the  construction of finitely generated projective
$C^* ( \Z^n, \omega_\theta )$-modules and follow the presentation in \cite{Li03}. Denote $\omega_\theta$ by $\omega$ and define a new cocycle $\omega_1$  on $\Z^n$ by $\omega_1(x, y)  = e(\langle \theta'x , y \rangle/2 )$, where $$
\theta'    =  
\left( \begin{array}{cc}
\theta_{11}^{-1}  &  -\theta_{11}^{-1}\theta_{12}\\
\theta_{21}\theta_{11}^{-1}  &  
\theta_{22} - \theta_{21}\theta_{11}^{-1}\theta_{12}
\end{array} \right).       
$$

Set $\mathcal{A} = C^* ( \Z^n, \omega )$ and
$\mathcal{B} = C^* ( \Z^n, \omega_1 )$.  Let $M$ be the group $\R^p \times \Z^q$, $G := M \times \hat{M}$ and $\langle \cdot , \cdot \rangle$ be the natural pairing between  $M$ and its dual group $\hat{M}$ (our notation does not distinguish between the pairing of a group and its dual group, and the standard inner product on linear spaces). 
Consider the Schwartz space $\mathcal{E}^\infty := \mathcal{S}(M)$
consisting of smooth and rapidly decreasing complex-valued functions on $M$.

Denote by $\mathcal{A}^\infty = \mathcal{S} (\Z^n, \omega)$ and  $\mathcal{B}^\infty = \mathcal{S} (\Z^n, \omega_1)$ 
the dense sub-algebras of $\mathcal{A} $ and $\mathcal{B} $, respectively.
Let us consider the following $(2p+2q) \times (2p + q)$ real valued matrix:
$$
 T = \left( \begin{array}{cc}
T_{11} & 0\\
0  & I_q\\
T_{31}  & T_{32}
\end{array} \right), 
$$ 
where $T_{11}$ is the invertible matrix such that $T_{11}^tJ_0T_{11} = \theta_{11}$, $J_0
  := \left( \begin{array}{ccc}
0 & I_p\\
  -I_p  & 0\\
  \end{array} \right)$, $T_{31} = \theta_{21}$ and $T_{32}$ is the matrix obtained from $\theta_{22}$ replacing the lower diagonal entries by zero. 
  
 We also define  the following $(2p+2q) \times (2p + q)$ real valued matrix:
 $$
S = \left( \begin{array}{cc}
J_0(T_{11}^t)^{-1} & -J_0(T_{11}^t)^{-1}T_{31}^t\\
0  & I_q\\
0  & T^t_{32}
\end{array} \right).
$$ 
Let 
$$
 J = \left( \begin{array}{ccc}
J_0 & 0 & 0\\
0 & 0 & I_q\\
0 & -I_q & 0\\
\end{array} \right) 
$$ and $J'$ be the matrix obtained from $J$ by replacing the negative entries of it by 0. Note that $T$ and $S$ can be thought as maps $\R^p \times \R^{*p}\times \Z^{q} \rightarrow G$ (see the definition 2.1 of the embedding map in \cite{Li03}).  
Let $P'$ and $P''$ be the canonical projections of $G$ to $M$ and $\hat{M}$, respectively, and let
$$T':=P'\circ T, \quad T'':=P''\circ T, \quad  S':= P'\circ S, \quad S'':= P''\circ S.$$  Then the following formulas define a $\mathcal{A}^\infty$-$\mathcal{B}^\infty$ bimodule structure on $\mathcal{E}^\infty$:

\begin{equation} \label{eq:proj_mod_T}
(f U_l^\theta)(x)= e(\langle -T(l), J^{\prime} T(l)/2\rangle)\langle x, T^{\prime \prime}(l) \rangle f(x-T^{\prime}(l)), 
\end{equation} 

\begin{equation} \label{eq:proj_mod_inner}
\langle f,g \rangle_{\mathcal{A}^\infty}(l)= e(\langle-T(l), J^{\prime} T(l)/2\rangle)\int_{G} \langle x,-T^{\prime \prime}(l) \rangle g(x+T^{\prime}(l))\bar{f}(x) dx, 
\end{equation}

\begin{equation} \label{eq:proj_mod_Tprime}
(V_l^\theta  f)(x)= e(\langle-S(l), J^{\prime} S(l)/2\rangle)\langle x, -S^{\prime \prime}(l) \rangle f(x+S^{\prime}(l)), 
\end{equation} 

\begin{equation} \label{eq:proj_mod_inner_Tprime}
{}_{\mathcal{B}^\infty}\langle f,g \rangle(l)= e(\langle S(l), J^{\prime} S(l)/2\rangle)\int_{G} \langle x,S^{\prime \prime}(l) \rangle \bar{g}(x+S^{\prime}(l))f(x) dx, 
\end{equation}
 where $U_l^\theta, V_l^\theta$ denote the canonical unitaries with respect to the group element $l \in \Z^n$ in $\mathcal{A}^\infty $ and $\mathcal{B}^\infty$, respectively.

See Proposition 2.2 in \cite{Li03} for the following well-known result.

\begin{thm}[Rieffel]\label{thm:rieffel_proj_mod}
The smooth module $\mathcal{E}^\infty$, with the above structures, is an $\mathcal{A}^\infty$-$\mathcal{B}^\infty$  Morita equivalence bimodule  which can be extended to a strong Morita equivalence between $\mathcal{A}$ and $\mathcal{B}$. 

\end{thm}
Let $\mathcal{E}$ be the completion of $\mathcal{E}^\infty$ with respect to the  $C^*$-valued inner products given above. Now  $\mathcal{E}$ 
becomes a right  projective $A$-module which is also finitely generated (see the discussion preceding Proposition 4.6 of \cite{ELPW}).  
The projective module corresponding to $q=0$ is called the {\it Bott class}. Note that this Bott class appears only for even dimensional tori.

\begin{rmk}\label{rmk:trace_of_heisenberg}
	The trace of the module $\mathcal{E}$, which was computed by Rieffel \cite{ri88}, is exactly the absolute value of the pfaffian of the upper left $2p \times 2p$ corner of the matrix $\theta,$ which is $\theta_{11}$. Indeed, as \cite[Proposition 4.3, page 289]{ri88} says that trace of $\mathcal{E}$ is $|\deta \widetilde{T}|$, where 
$$
\widetilde{T}= \left( \begin{array}{cc}
T_{11} & 0\\
0  & I_q\\
\end{array} \right),
$$ the relation $T_{11}^tJ_0T_{11} = \theta_{11}$ and the fact $\deta (J_0) =1$ give the claim.
\end{rmk}


\section{A quick look into noncommutative orbifolds}

Let $W:=(a_{ij})$ be an $n \times n$ matrix of finite order with integer entries acting on $\Z^n$ and $F$ be the cyclic group generated by $W$. In addition, we assume that $W$ is a $\theta$-symplectic matrix as in section 2. Hence $F$ is a finite subgroup of $SP(n, \Z, \theta):= \{ A \in GL(n,\Z) : A^T\theta A = \theta$\}. By Lemma 2.1 of \cite{ELPW} we have $C^*(\Z^n\rtimes F, \omega_\theta ') = A_{\theta} \rtimes_\alpha F$ with respect to the action (\cite[Equation 2.6]{jh}) :
 \begin{equation}\label{eq:actionontori}
 	\alpha(U_i)= e(\sum_{k=2}^n\sum_{j=1}^{k-1}a_{ki}a_{ji}\theta_{jk})U_1^{a_{1i}}\cdots U_n^{a_{ni}},
 \end{equation} where $U_1,...,U_n$ are the generators of $A_\theta$.

Let us look into the case where $n=2$. Note that $SP(2, \Z, \theta) = SL(2, \Z)$. Finite cyclic subgroups of $SL(2,\Z)$ are up to conjugacy generated by the following 4 matrices:  
$$W_2
  := \left( \begin{array}{ccc}
-1 & 0\\
  0  & -1\\
  \end{array} \right) , \:  
	W_3
  := \left( \begin{array}{ccc}
-1 & -1\\
  1  & 0\\
  \end{array} \right), 
  $$
  
  $$ W_4
  := \left( \begin{array}{ccc}
0 & -1\\
  1  & 0\\
  \end{array} \right),
  \:  W_6
  := \left( \begin{array}{ccc}
0 & -1\\
  1  & 1\\
  \end{array} \right),
   $$ where the notation $W_r$ indicates that it is a matrix of order $r$.
   
   The actions of these matrices are considered already in \cite{ELPW}, where the authors constructed projective modules over the corresponding crossed products and used these projective modules to prove some classification results for these crossed products. 
   
   For $n\ge 3$ finding finite order matrix $W\in SP(n, \Z, \theta)$ is non-trivial. For $n=3$, there is only one such matrix ($-I_3$) acting on all $A_\theta$'s.  In \cite{jh} the authors found some $W$'s and associated actions for $n \ge 4$ such that the crossed products  are well defined. 
   
\section{Projective modules over noncommutative orbifolds}

   One natural question is how does one extend the projective modules over noncommutative tori to the aforementioned crossed products? Our main theorem addresses this question for the Bott classes.

In the following sections (except Section 7) we consider $n$ to be an even number, $n=2m$. 
Suppose $F = \langle W \rangle$ is a finite cyclic group acting on $\Z^n$. We want to build some projective modules over $C^*(\Z^n\rtimes F, \omega_\theta ')$. Note that $W$ needs to be a $\theta$-symplectic matrix, i.e $W^T\theta W = \theta$, as noted earlier.

In order to construct projective modules over $C^*(\Z^n\rtimes F, \omega_\theta ')$, we will use the so-called metaplectic representation of the symplectic matrix $W$.
When $\theta$ is the standard skew-symmetric matrix $J$ then $W$ is also a standard symplectic matrix. 

We denote the group of all  $J$-symplectic matrices (also known as standard symplectic matrices) by $\mathcal{SP}(n)$, which is called the symplectic group. We refer to Chapter 2 of \cite{Go} for preliminaries on symplectic groups and their metaplectic representations. We recall the metaplectic action associated to the symplectic matrix $W$. Any symplectic matrix can be written as product of two free symplectic matrices (see page 38, \cite{Go}) which is by definition a symplectic matrix $$ \left( \begin{array}{ccc}
A & B\\
  C  & D\\
 
\end{array} \right) ,
$$
such that $\det(B) \neq 0$. Let $W$ to be a free symplectic matrix.
 We now associate to $W$ the {\it generating function}:
\begin{equation}\label{eq:gen} W(x,x') = \frac{1}{2}\langle DB^{-1} x, x\rangle-\langle B^{-1} x, x'\rangle + \frac{1}{2}\langle B^{-1}A x', x'\rangle,
\end{equation} when $x,x' \in \R^m$.

In what follows, for $y\in\Q$, we denote the complex number  $i^y$ for some choice of element in the range of $i^y$ for the multivalued function $i^z, z\in \C$. In other words, we fix a branch for $i^z$. 
\begin{dfn}
The {\it metaplectic operator} (metaplectic transformation) associated to $W$ on $\mathscr{S}(\R^m)$ is given by

$$\mathcal{F}_Wf(x) =   i^{s-\frac{m}{2}} \sqrt{|\det(B^{-1})|}\int_{\R^m} e(W(x,x'))f(x') dx';$$ 
the integer $s$ (sometimes called Maslov index) corresponds to a choice of the argument $\arg$ of $\det B^{-1}$:

$$s\pi \equiv \arg \; (\det B^{-1})\quad \text{mod 2$\pi$.}  $$ 
\end{dfn}

These operators can be extended to $L^2(\R^m)$ giving unitary operators on $L^2(\R^m)$ (see page 81, \cite{Go}). We denote by $\mathcal{MP}(n)$ the group of metaplectic operators which is a subgroup of the  group of unitary operators of  $L^2(\R^m)$.
\begin{thm}

There exists an exact sequence:
\begin{center}
	\begin{tikzcd}
  0 \arrow{r} & \Z_2 \arrow{r} & \mathcal{MP}(n) \arrow{r}& \mathcal{SP}(n)           \arrow{r} & 0 ,
    \end{tikzcd}
    \end{center}
    where the map $\mathcal{MP}(n) \rightarrow \mathcal{SP}(n)$ is uniquely determined by the map $\mathcal{F}_W \rightarrow W$.
    	
\end{thm}

\begin{proof}
	See page 84, \cite{Go}.
\end{proof}
	
	One also defines the circle extension of $\mathcal{SP}(n)$,  $\mathcal{MP}^c(n)$. This is   defined to be the group  $ \mathcal{MP}(n) \times_{\Z_2} \mathbb{S}^1: (\mathcal{MP}(n) \times \mathbb{S}^1) / \Delta(\Z_2),$ $\Delta(\Z_2)$ being the diagonal $\Z_2 \times \Z_2$ sitting inside $\mathcal{MP}(n) \times \mathbb{S}^1.$ This gives rise to the exact sequence 
	\begin{center}
		\begin{tikzcd}
  0 \arrow{r} & \mathbb{S}^1 \arrow{r} & \mathcal{MP}^c(n) \arrow{r}& \mathcal{SP}(n)           \arrow{r} & 0 ,
    \end{tikzcd}
    	\end{center} where $\mathbb{S}^1$ denotes the circle group.
    	
In the following, we shall often write $fW$ for $\mathcal{F}_W (f)$.

Following \cite[Section 3.2.2]{Go} the following matrices generate (as a group) all $J$-symplectic matrices: 

$$J
  := \left( \begin{array}{ccc}
0 & I\\
  -I  & 0\\
  \end{array} \right) , \:  M_L
  := \left( \begin{array}{ccc}
L & 0\\
  0  & (L^{T})^{-1}\\
  \end{array} \right), \: V_P
  := \left( \begin{array}{ccc}
I & 0\\
  P  & I\\
  \end{array} \right),
   $$
  for a symmetric $m \times m$ matrix $P$ and an invertible $m \times m$ matrix $L$. 
  
Following \cite[Section 7.1.2]{Go}, we write down the metaplectic operators  (up to some constant which will not matter in the proof) corresponding to $J$, $M_L$ and $V_P$:
  
 \begin{equation} \label{eq:fourier1}
(f\,J)(x)  =   \int_{\R^m} e(\langle -x,x' \rangle)f(x') dx'\end{equation}
\begin{equation} \label{eq:flip1}
(f\,M_L)(x)  = \sqrt{det(L)}f(L(x))\end{equation}
\begin{equation} \label{eq:otheraction1}
(f\,V_P)(x)  = e(\frac{1}{2}\langle Px,x \rangle)f(x).\end{equation}

Hence it suffices to check statements of multiplicative type about metaplectic transformations for the above three operators and also note that the Schwartz space is invariant under metaplectic transformations, see \cite[Corollary 63]{Go}.  For our result we always assume $\theta$  to be a non-degenerate matrix.   
\\

We recall the following proposition from \cite[Proposition 4.5]{ELPW}.

\begin{prp}\label{mainprop}
	Suppose $F$ is a finite group acting on a C*-algebra $A$ by the action $\alpha$. Also suppose that $\mathcal{E}$  is a finitely generated projective (right) $A$-module with a right action $T : F \rightarrow \Aut(\mathcal{E})$, written $(\xi, g) \rightarrowtail \xi T_g$, such that $\xi (T_g)a = (\xi\alpha_g(a)) T_g$ for all $\xi \in  \mathcal{E}, a \in A,$ and $g \in F$. Then $\mathcal{E}$ becomes a finitely generated projective $A\rtimes F$ module with action defined by
	
$$  \xi \cdot (\sum_{g \in F} a_g\delta_g) =  \sum_{g \in F} (\xi a_g)T_g $$.

Also, if we restrict the new module to $A$, we get the original $A$-module $\mathcal{E}$, with the action of $F$ forgotten.
\end{prp}

\begin{proof}
	This is exactly the construction of Green--Julg map. For a proof, see \cite[Proposition 4.5]{ELPW}.
\end{proof}

Now we are in the position to formulate our main theorem. Let $\mathscr{S}(\R^m)$ be as in Section 3.

\begin{thm} \label{main}

Let $W$ be the generator of the finite cyclic group $F$ acting on $\Z^n$ with $W^T\theta W = \theta$ and hence on $C^*\mathcal(\Z^n,\omega_{\theta})$. Then the metaplectic action of $W$ on $\mathscr{S}(\R^m)$ extends to an action on $\mathcal{E}$ such that $\mathcal{E}$ becomes an $F$-equivariantly finitely generated projective $C^*(\Z^n,\omega_{\theta})$ module and thus a finitely generated  projective module over $C^*(\Z^n,\omega_{\theta}) \rtimes F$.
\end{thm}
\begin{proof}  
We divide the proof in two parts.\\
{\bf First part:} (the case $\theta = -J$): Recall that from \eqref{eq:proj_mod_T} for the choice of $T
  := \left( \begin{array}{ccc}
-I & 0\\
  0  & I\\
  \end{array} \right)$  the action of $\mathscr{S}(\Z^n, \omega_{-J}) = C^\infty(\T^n)$ on $\mathscr{S}(\R^m)$  is given by the following: 

\begin{equation}
	fU_{i}^p(y_1,y_2,\dots,y_m) = f(y_1,y_2,\dots,y_{i}+p,\dots,y_m) , \: \text{if} \:    i\leq m,\end{equation}
\begin{equation}fU_{i}^p(y_1,y_2,\dots,y_m) = e(py_{i-m})f(y_1,y_2,\dots,y_m) , \: \text{if} \:    i> m, \end{equation}
where $U_i$'s are the generators of $n$-dimensional smooth torus $C^\infty(\T^n)$.
[Note that $\theta = -J$ is chosen instead of $\theta = J$ to keep the formulas somewhat similar to \cite{ELPW}]. Let $\alpha_{W}$ denotes the action of the matrix $X$ on $\mathscr{S}(\Z^n)$: $\alpha_{W}(\phi) (x) = \phi(W^{-1}x)$.
According to Proposition \ref{mainprop}, we still have to check the following equation to complete the proof:

\begin{equation} \label{eq:main}
f(W)\phi = (f\alpha_W(\phi))W,
\end{equation}
for all $f\in \mathscr{S}(\R^m)$ and $\phi \in \mathscr{S}(\Z^n,\omega_{-J})$, which will then imply that $\mathcal{E}$ becomes $F$-equivariant. 
Also, since $\mathscr{S}(\Z^n,\omega_{-J})$ is generated by $U_1,U_2,\dots, U_n$, it is enough to check \eqref{eq:main} for $\phi = U_1,U_2,\dots, U_n$.

So we are left with checking the following equations:
\begin{equation} \label{eq:fourier0}
fJU_i = (f\alpha_J(U_i))J,\end{equation}
\begin{equation} \label{eq:flip0}
fM_LU_i = (f\alpha_{M_L}(U_i))M_L,\end{equation}
\begin{equation} \label{eq:otheraction0}
fV_PU_i = (f\alpha_{V_P}(U_i))V_P ,\end{equation}
for all $1\leq i \leq n$. 

First we check the equations  \eqref{eq:fourier0},  \eqref{eq:flip0} and  \eqref{eq:otheraction0}  for $1\leq i \leq m$. 
The left hand side (LHS) of  \eqref{eq:fourier0} is

{\tiny\begin{align*}
(fJU_i)(x_1,x_2,\dots, x_m) &= (fJ)(x_1,x_2,\dots,x_i+1\dots, x_m) ,
\\ &= \int_{\R^m} e(-\langle (x_1,x_2,\dots,x_i+1\dots, x_m), (x'_1,x'_2,\dots, x'_m)\rangle)f(x') dx',
\\
&= \int_{\R^m} e(-\langle (x_1,x_2,\dots, x_m), (x'_1,x'_2,\dots, x'_m)\rangle).e(-x'_{i})f(x') dx';
\end{align*}}
and the right hand side (RHS):

{\tiny
\begin{align*}
(f\alpha_J(U_i))J(x_1,x_2,\dots, x_m) &= \int_{\R^m} e(-\langle (x_1,x_2,\dots, x_m), (x'_1,x'_2,\dots, x'_m)\rangle)f\alpha_J(U_i)(x') dx',
\\&= \int_{\R^m} e(-\langle (x_1,x_2,\dots, x_m), (x'_1,x'_2,\dots, x'_m)\rangle)(fU_{i+m}^{-1})(x') dx',
\\&= \int_{\R^m} e(-\langle (x_1,x_2,\dots, x_m), (x'_1,x'_2,\dots, x'_m)\rangle).e(-x'_{i})f(x') dx'.
\end{align*}}

{\setlength{\parindent}{0cm}
Hence we have proved \eqref{eq:fourier0}.}
The LHS of \eqref{eq:flip0} equals  
\begin{align*}
(fM_LU_i)(x_1,x_2,\dots, x_m) &= (fM_L)(x_1,x_2,\dots,x_i+1,\dots , x_m),
\\ &= \sqrt{det(L)}f(L(x_1,x_2,\dots,x_i+1,\dots , x_n));
\end{align*}
and the RHS is
\begin{align*}
(f\alpha_{M_L}(U_i))M_L(x_1,x_2,\dots, x_m) &= \sqrt{det(L)}(f\alpha_{M_L}(U_i))L(x_1,x_2,\dots, x_m),
\\ &= \sqrt{det(L)}f(L(x_1,x_2,\dots, x_m)+L(x_i)),
\\
&= \sqrt{det(L)}f(L(x_1,x_2,\dots,x_i+1,\dots , x_m)).
\end{align*}
Hence we have demonstrated \eqref{eq:flip0}.
We have for the LHS \eqref{eq:otheraction0}: 
\begin{align*}
(fV_PU_i )(x_1,x_2,\dots, x_m) &= (fV_P)(x_1,x_2,\dots,x_i+1,\dots, x_m),
\\& = e(\frac{1}{2}\langle P(x_1,x_2,\dots,x_i+1,\dots, x_m),\cdots
\\  &\qquad    (x_1,x_2,\dots,x_i+1,\dots, x_m)\rangle)\cdots \\& \qquad  f(x_1,x_2,\dots,x_i+1,\dots, x_m),
\end{align*}and the RHS is
\begin{align*}
(f\alpha_{V_P}(U_i))V_P(x_1,x_2,\dots, x_m) &= e(\frac{1}{2}(Px \cdot x))(f\alpha_{V_P}(U_i))(x),
\\ &= e(\frac{1}{2}\langle P(x_1,x_2,\dots,x_i+1,\dots, x_m),\cdots
\\ &\qquad   (x_1,x_2,\dots,x_i+1,\dots, x_m)\rangle)\cdots \\ &\qquad   f(x_1,x_2,\dots,x_i+1,\dots, x_m).
\end{align*}
Hence we have shown \eqref{eq:otheraction0}.

 Now, let $m< i \leq n$. We check the equations  \eqref{eq:fourier0},  \eqref{eq:flip0} and  \eqref{eq:otheraction0} for these values of $i$ . 
\\ For \eqref{eq:fourier0} the LHS is 
\begin{align*}
(fJU_i)(x_1,x_2,\dots, x_m) &= e(x_{i-m})(fJ)(x_1,x_2,\dots, x_m),
\\ &=  e(x_{i-m}) \int_{\R^m} e(-\langle x,x' \rangle)f(x') dx';
\end{align*}
and the RHS 
\begin{tiny}
\begin{align*}
(f\alpha_J(U_i))J(x_1,x_2,\dots, x_m) &=  \int_{\R^m} e(-\langle x,x' \rangle)(f\alpha_J(U_i))(x') dx',
\\ &= \int_{\R^m} e(-\langle x,x' \rangle)(f(U_{i-m}))(x') dx',
\\&= \int_{\R^m} e(-\langle x,x' \rangle)f(x'_1,x'_2,\dots,x'_{i-m}+1,\dots , x'_m) dx',
\\&= e(x_{i-m}) \int_{\R^m} e(-\langle x,x' \rangle)f(x'_1,x'_2,\dots,x'_{i-m},\dots , x'_m) dx', 
\\ &=  e(x_{i-m}) \int_{\R^m} e(-\langle x,x' \rangle)f(x') dx'.
\end{align*}
\end{tiny}

{\setlength{\parindent}{0cm}Hence we have proved \eqref{eq:fourier0}.}
\\ For \eqref{eq:flip0}, the LHS
\begin{align*}
(fM_LU_i)(x_1,x_2,\dots, x_m) &= e(x_{i-m})(fM_L)(x_1,x_2,\dots, x_m),
\\
&= \sqrt{det(L)}e(x_{i-m})f(L(x_1,x_2,\dots, x_m));
\end{align*}
and the RHS 

\begin{small}
\begin{align*}
(f\alpha_{M_L}(U_i))M_L(x_1,x_2,\dots, x_m) &= \sqrt{det(L)}(f\alpha_{M_L}(U_i))(L(x_1,x_2,\dots, x_m)),
\\ &= \sqrt{det(L)}e(\langle (L^{-1})^T(e_{i-m}), L(x_1,x_2,\dots, x_m)\rangle) \cdots \\ & \qquad f(L(x_1,x_2,\dots, x_m)),
\\
&= \sqrt{det(L)}e(\langle e_{i-m}, (L^{-1}L)(x_1,x_2,\dots, x_m)\rangle) \cdots
\\ & \qquad f(L(x_1,x_2,\dots, x_m)),
\\
&= \sqrt{det(L)}e(x_{i-m})f(L(x_1,x_2,\dots, x_m)). 
\end{align*} \end{small} 

{\setlength{\parindent}{0cm}Thus \eqref{eq:flip0} is verified.}
\\ For \eqref{eq:otheraction0}, the LHS 
\begin{align*}
(fV_PU_i )(x_1,x_2,\dots, x_m)&= e({x_{i-m}})(fV_P)(x_1,x_2,\dots, x_m),
\\ &= e({x_{i-m}})e(\frac{1}{2} \langle Px,x \rangle)f(x);
\end{align*}
and the RHS:
\begin{align*}
(f\alpha_{V_P}(U_i))V_P(x_1,x_2,\dots, x_m)  &= e(\frac{1}{2}\langle Px,x \rangle)(f\alpha (U_i)(x),
\\ &= e(\frac{1}{2}\langle Px,x \rangle)(fU_i)(x),
\\
&= e({x_{i-m}})e(\frac{1}{2} \langle Px,x \rangle)f(x).
\end{align*}
Hence we have proved \eqref{eq:otheraction0}.

Now we have the following diagram:
\begin{center}

\begin{tikzcd}
                &                                 &                                   & F                 \arrow{d}          &   \\
    0 \arrow{r} & \Z_2 \arrow{r}   & \mathcal{MP}(n)                 \arrow{r}        & \mathcal{SP}(n)           \arrow{r} & 0 .\\
    
\end{tikzcd}
\end{center}
In the above diagram it is not assured that the inclusion $F \hookrightarrow \mathcal{SP}(n)$ lifts to an inclusion $F \hookrightarrow \mathcal{MP}(n)$. Since $F$ is cyclic the following lift is always possible:

\begin{center}

\begin{tikzcd}
                &                                 &                                   & F  \arrow[ld, dashrightarrow]               \arrow{d}          &   \\
    0 \arrow{r} & \mathbb{S}^1 \arrow{r}   & \mathcal{MP}^c(n)                 \arrow{r}        & \mathcal{SP}(n)           \arrow{r} & 0 ,\\
    
\end{tikzcd}

\end{center}
where $\mathcal{MP}^c(n)$ is the circle extension of $\mathcal{SP}(n)$. Indeed, for the generator $W \in F$, we can choose a scalar $z \in \T$ to get that the order of the operator $z\cdot \mathcal{F}_W \in \mathcal{MP}^c(n)$ is same as the order of the element $W \in F$. Hence the  inclusion $F \hookrightarrow \mathcal{MP}^c(n)$ gives the required action of $W$ on $\mathscr{S}(\R^m)$.
\\~\\
{\bf Second part (the general case):}

Let $\theta$ be a general non-degenerate skew-symmetric matrix. In this case $W_\theta^T\theta W_\theta = \theta$. We recall the construction of the projective modules in this case. Since $\theta$ is non-degenerate, there exists an invertible matrix $T$ such that $T^TJT = \theta$. Recall that the action of $U_l^\theta$ (for $l \in \Z^n$) on $\mathscr{S}(\R ^m)$ is defined by 

\begin{equation} \label{theTmap}
(fU_l^\theta)(x)= e((-T(l)\cdot J^{\prime} T(l)/2))e(\langle x, T^{\prime \prime}(l) \rangle) f(x-T^{\prime}(l)). \end{equation}.

First we note that $W :=TW_\theta T^{-1}$ is a $J$-symplectic matrix (a matrix $A$ is $J$-symplectic if $A^T JA = J$). Thus we can define 
$$fW=f(TW_\theta T^{-1}) := \mathcal{F}_{TW_\theta T^{-1}}(f)$$ 
for $f\in \mathscr{S}(\R^m)$ and $fW_\theta$ to be the function $fW$. Consequently, in this case we have to check the following equation:

\begin{equation}\label{main-general}
(fW_\theta)U_l^\theta(x) = (f\alpha_{W_\theta}(U_l^\theta))W_\theta(x), \hspace{.2cm} x\in\R^m.
	\end{equation}

\begin{align*}
(fW_\theta)U_l^\theta(x) &= e(\langle-T(l), J^{\prime} T(l)/2\rangle)e(\langle x, T^{\prime \prime}(l) \rangle) (fW_\theta)(x-T^{\prime}(l)) ,
\\ &= e(\langle-T(l), J^{\prime} T(l)/2\rangle)e(\langle x, T^{\prime \prime}(l) \rangle) (fW)(x-T^{\prime}(l)),
\\ &= e(\langle-T(l), J^{\prime} \Id(Tl)/2\rangle)e(\langle x, \Id^{\prime \prime}(Tl) \rangle) (fW)(x-\Id^{\prime}(Tl)),
\\ &= (fW)U_{Tl}^J(x),
\\ & \underset{}{=} (f\alpha_W(U_{Tl}^J))W(x) \qquad  \text{(using \eqref{eq:otheraction0})}
\end{align*}
and the RHS 
\begin{align*}
(f\alpha_{W_\theta}(U_l^\theta))W_\theta(x)  &= (f\alpha_{W_\theta}(U_l^\theta))W(x),
\\ &= \int_{\R^m} e(W(x,x'))(f\alpha_{W_\theta}(U_l^\theta))(x') dx',
\\
&= \int_{\R^m} e(W(x,x'))(f(U_{W_\theta(l)}^\theta))(x') dx',
\\
&= \int_{\R^m} e(W(x,x'))(f(U_{W(Tl)}^J))(x') dx', \qquad  \text{(using \eqref{eq:change})} 
\\ &= \int_{\R^m} e(W(x,x'))(f\alpha_W(U_{Tl}^J))(x') dx',
\end{align*}
where
\begin{equation}\label{eq:change}
(f(U_{W_\theta(l)}^\theta))(x') = (f(U_{W(Tl)}^J))(x'), \qquad  x'\in\R^m
	\end{equation}
because of the fact:

{\tiny\[ (f(U_{W_\theta(l)}^\theta))(x')  = e(\langle-T(W_\theta(l)), J^{\prime} T(W_\theta(l))/2\rangle)e(\langle x', T^{\prime \prime}(W_\theta(l)) \rangle) f(x'-T^{\prime}(W_\theta(l))),\]}
which is equal to
{\tiny
\begin{eqnarray*}
&&e(\langle-T(W_\theta(T^{-1}Tl)), J^{\prime} T(W_\theta(T^{-1}Tl))/2\rangle)e(\langle x', T^{\prime \prime}(W_\theta(T^{-1}Tl)) \rangle) f(x'-T^{\prime}(W_\theta(T^{-1}Tl))),\\
&=&e(\langle-(TW_\theta T^{-1})(Tl), J^{\prime} (TW_\theta T^{-1})(Tl)/2\rangle)e(\langle x', (TW_\theta T^{-1})''(Tl) \rangle)  f(x'-(TW_\theta T^{-1})'(Tl)),\\
&=&e(\langle-W(Tl), J^{\prime} W(Tl)/2\rangle)e(\langle x', \Id''(W(Tl)) \rangle)  f(x'-\Id'(W(Tl))),\\
&=&(f(U_{W(Tl)}^J))(x').
\end{eqnarray*}}

We finish the proof with the compatibility of the action with the $\langle .,. \rangle_{\mathcal{A}^\infty}$ as defined in \eqref{eq:proj_mod_inner}:
\begin{equation*}
  \langle fW_\theta,gW_\theta \rangle_{\mathcal{A}^\infty}=\alpha_{W_\theta^{-1}}(\langle f,g \rangle_{\mathcal{A}^\infty}).
\end{equation*}
Replacing $f$ by $fW_\theta$, it suffices to check:
\begin{equation}\label{eq:comp_inner}
   \langle f,gW_\theta \rangle_{\mathcal{A}^\infty}=\alpha_{W_\theta^{-1}}(\langle fW_\theta^{-1},g \rangle_{\mathcal{A}^\infty}).
\end{equation}
The argument is based on some observations: (i) the 
explicit description of $\langle.,.\rangle_{\mathcal{A}^\infty}$ in terms of the right action of 
$\mathcal{A}^\infty$ on $\mathscr{S}(\R^m)$:  
\[\langle f,g \rangle_{\mathcal{A}^\infty}(l)=\langle gU^\theta_{-l},f\rangle_{L^2}
,\]
for $\langle f,g\rangle_{L^2}=\int_{\R^m}f(x)\overline{g(x)}dx$, and (ii) the relations:
\[\alpha_{W_\theta^{-1}}(\langle f,g \rangle_{\mathcal{A}^\infty})(l)=\langle g\alpha_{W_\theta}(U^\theta_{-l}),f\rangle_{L^2}
.\]

The realization of $\langle.,.\rangle_{\mathcal{A}^\infty}$ in terms of the right action allows us to use equation \eqref{main-general}: 
\[(fW_\theta)U_l^\theta(x) = (f\alpha_{W_\theta}(U_l^\theta))W_\theta(x), \hspace{.2cm} x\in\R^m\] in the proof of \eqref{eq:comp_inner}:
\begin{eqnarray*}
 \langle f,gW_\theta \rangle_{\mathcal{A}^\infty}(l)&=&\langle (gW_\theta)U^\theta_{-l},f\rangle_{L^2},\\
                                                 &=&\int_{\mathbb{R}^m} (g\alpha_{W_\theta}(U_{-l}^\theta)W_\theta(x))\overline{f(x)}dx,\\
																								 &=&\int_{\mathbb{R}^m} (g\alpha_{W_\theta}(U_{-l}^\theta))(x)\overline{(fW_\theta^{-1})(x))}dx,\\
																								 &=&\alpha_{W_\theta^{-1}}(\int_{\mathbb{R}^m} (gU_{-l}^\theta)(x)\overline{(fW_\theta^{-1})(x))}dx),\\
																								 &=&\alpha_{W_\theta^{-1}}(\langle fW_\theta^{-1} ,g\rangle_{\mathcal{A}^\infty}),
\end{eqnarray*}
which is the desired identity.
\end{proof} 
Note that we have always written the operator $T_W$ up to some constant of modulus one, which is not essential for our discussion.

\section{The 2-dimensional case - revisited}

The results for the 2-dimensional case \cite{ELPW} are revisited from the perspective of metaplectic transformations. As mentioned before, there are up to conjugation four matrices of  finite order in $\mathrm{SL}_2(\Z)$ generating $\Z_2, \Z_3, \Z_4, \Z_6$. For  the $\Z_2$ action on $\mathscr{S}(\R)$, given by $f \rightarrow \tilde{f}$, where $\tilde{f}(x) = f(-x)$, the corresponding module, called the flip module, over $A_\theta \rtimes \Z_2$ is quite well studied by Walters \cite{wa01-2}. In the next section we discuss in more detail flip modules in the higher-dimensional setting. The $\Z_4$ action is given by the Fourier automorphism $f \rightarrow \tilde{f}$ where $\tilde{f}(x) =  \int_{\R} e(\langle x,x' \rangle)f(x')dx'$. Walters has studied these modules extensively and among other things he computed the Chern character for the flip modules and Fourier modules. The $\Z_3$ and $\Z_6$ actions are similar so we only treat the $\Z_6$ action.

The cyclic group $\Z_6$ is generated by the matrix  $W_6
  := \left( \begin{array}{ccc}
1 & -1\\
  1  & 0\\
  \end{array} \right)$ that we denote by $W$.  Note that this $W_6$ slightly differs from $W_6$ from section 4. We choose this $W_6$ to keep the final formula similar to the formula for $W_6
  := \left( \begin{array}{ccc}
0 & -1\\
  1  & 1\\
  \end{array} \right)$ in \cite{ELPW}. One should note that the action of finite group on the projective module as in Proposition \ref{mainprop} is not unique.
  
     The generating function associated to $W = W_6$ is given by 
  
  $$W(x,x') = xx' - \frac{1}{2}x'^2,$$ 
 which follows from \eqref{eq:gen}.
  The corresponding metaplectic transformation (for the choice $s=1$) is
  $$\mathcal{F}_W(f)(x) =   \sqrt{i}\int_{\R} e(xx' - \frac{1}{2}x'^2)f(x') dx',\,  f\in \mathscr{S}(\R)$$
  
  The following proposition is due to Walters:
  
  \begin{prp}
  	$$(\mathcal{F}_W)^6 = - I.$$
  \end{prp}
  We modify the operator $\mathcal{F}_W$ to $i^{-\frac{1}{3}}\mathcal{F}_W$, which amounts to including the Maslov index of the transformation. Then $\mathcal{F}_W^6= I$. The corresponding projective module over $A_J \rtimes \Z_6$ is called the hexic module by Walters, where $J= \left( \begin{array}{ccc}
0 & 1\\
  -1  & 0\\
  \end{array} \right)$ is the standard symplectic form on $\R^2$. For a general $A_\theta$, choosing $T
  := \left( \begin{array}{ccc}
-\theta & 0\\
  0  & 1\\
  \end{array} \right)$, we get from the main theorem: 
   $$\mathcal{F}_{W_{\theta}}(f)(x) =   i^{\frac{1}{6}}\theta^{-\frac{1}{2}}\int_{\R} e(\frac{1}{2\theta}(2xx' -x'^2))f(x') dx',  f\in \mathscr{S}(\R);$$
   which is the exactly the formula for the $\Z_6$ action considered in \cite{ELPW}.

\section{$\K$-theory of the \cp\ of $n$-dimensional \nt\ with the flip action }\label{Sec:ntflip}

 We consider the $n \times n$ matrix $W = -I_n$ which generates the two element group. Suppose this group acts on a $n = 2p + q$-dimensional noncommutative torus with respect to the parameter $\theta$ with $\theta 
  := \left( \begin{array}{ccc}
\theta_{11} & \theta_{12}\\
\theta_{21} & \theta_{22}\\
  \end{array} \right)$, $\theta_{11}$ being the left $2p \times 2p$ corner, which amounts to the condition $W^T\theta W = \theta$ that holds in this case. We call this action the flip action.

  We define the following operator on  $\mathscr{S}(\R^p \times \Z^q)$ with respect to $W$:
  \begin{equation}
	T_W (f)(x,t):= f(-x,-t).\end{equation}
 $\mathscr{S}(\R^p \times \Z^q)$ with respect to this action is a $A_\theta ^\infty \rtimes \Z_2$ module which can be completed to an $A_\theta \rtimes \Z_2$ module. For a totally irrational $3 \times 3$ skew-symmetric matrix $\theta$ (the definition of which we will introduce shortly),  we will see that any generator of $\K_0(A_{\theta})$ can be given by completions of modules of the type $\mathscr{S}(\R^p \times \Z^q)$.   Hence it will follow that all the generators of $\K_0(A_{\theta})$ can be extended to provide classes in $\K_0$ of  $A_\theta \rtimes \Z_2$.  We will also show that $\K$-theory classes of these modules can be extended to a generating set   of $\K_0(A_{\theta}\rtimes \Z_2)$ for 3-dimensional noncommutative tori. Our results will show that this should also be the case for the general $n$-dimensional case, but at this moment we are unable to compute the generators of $\K_0(A_{\theta}\rtimes \Z_2)$ for $n$-dimensional noncommutative tori $A_\theta.$ It should be noted that $\K_1(A_{\theta}\rtimes \Z_2)$ is trivial (\cite{ELPW}).

Let  $\te$ be a real skew-symmetric $n \times n$ matrix  and $\theta'$ be  the  upper left $(n-1) \times (n-1)$ block of $\theta$. In this case, $A_{\te}$ can be written as a crossed product $ A_{\theta'}\rtimes \Z$, where the action $\gamma$  of $\Z$ on $A_{\theta'}$ given (by the generator of $\Z$ ) by
$ U_i \rightarrow e(\theta_{in})U_i,$ for $i =1, \cdots, n-1.$  Now $A_{\theta}\rtimes \Z_2 = A_{\theta'}\rtimes \Z \rtimes \Z_2 = A_{\theta'}\rtimes \Z_2 *  \Z_2$, since $\Z_2 *  \Z_2$ is isomorphic to $\Z \rtimes \Z_2$ as groups (see \cite[Proposition 6]{FW} for more details). Note that one copy of $\Z_2$ acts on $A_{\theta'}$ by flip action $\beta$ and the other by $\alpha = \gamma \circ \beta$. Our next step is to understand the $\K$-theory of $A_{\theta'}\rtimes (\mathbb{Z}_2 * \mathbb{Z}_2 )$. 


 For a general crossed product $A \rtimes \Z_2$, we first define a map $p$ which goes from $A \rtimes \Z_2$ to $M_2(A)$  such that $$ p(a + bW) = \left( \begin{array}{ccc}
a & b\\
WbW& WaW\\
  \end{array} \right),$$ where $W$ is the unitary in $A \rtimes \Z_2$ implementing the action of the generator of $\Z_2$. This induces a map $p_*: \K_0(A \rtimes \Z_2) \rightarrow \K_0(A),$ which is known to be the inverse of Green--Julg map.
We recall the six term exact sequence by Natsume \cite{Na} which was used by Farsi--Watling \cite{FW} to compute the $\K$-theory of $A_{\theta'}\rtimes (\mathbb{Z}_2 * \mathbb{Z}_2 )$.
For a free product $H_1*H_2$ acting on a \ca\ $A$, Natsume obtained the following exact sequence:

{\tiny \[
\begin{CD}
\K_0 (A) @>{i_{1*}-i_{2*}}>>\K_0 (A\rtimes H_1)\oplus \K_0 (A\rtimes H_2)@>{j_{1*}+j_{2*}}>> \K_0 (A\rtimes H_1*H_2)   \\
@A{}AA & &  @VV{e_1}V            \\
 \K_1 (A\rtimes H_1*H_2) @<<{j_{1*}+j_{2*}}<  \K_1 (A\rtimes H_1)\oplus \K_1 (A\rtimes H_2) @<<{i_{1*}-i_{2*}}< \K_1 (A),
\end{CD}
\]}

\noindent where $i_1, i_2, j_1, j_2$ are the natural inclusion maps. 
The right vertical map $e_1$, which we will describe in a while, is constructed in Natsume's paper. We call it exponential map since it is based on the exponential map in $\K$-theory. We want to compare the above sequence with the six-term exact sequence obtained from the classical Toeplitz exact sequence (with coefficient in $A$) which is same as the Pimsner--Voiculescu exact sequence for actions of $\Z$ on the C*-algebra $A$.

From the definition of the crossed product, any crossed product algebra, $A\rtimes_\alpha G$, for a unital C*-algebra $A$ and a discrete group $G$, has a natural representation (also called regular representation) $\iota$ on the Hilbert module $l^2( G, A)$ which is given by  $$\iota(a)(\xi) (g) = \alpha_{g^{-1}}(a)\xi (g), \quad \iota(h)(\xi) (g) = \xi (h^{-1}g),$$ for $a \in A$ and $g,h \in G$. 
Let $\Z$ act on a unital C*-algebra $A$ by an action $\gamma$. The classical Toeplitz algebra $\mathcal{T}^A$ with coefficients in $A$ is defined as follows: we restrict the natural representation $\iota$ of $A$ from $l^2( \Z, A)$ to $l^2( \Ne, A)$ (note that the restriction is well defined). Call this restricted representation $\iota_1$. When there is no confusion, we just call $\iota(a)$ and $\iota_1(a)$ by $a$.  Take the right shift operator $S$ on $l^2( \Ne, A)$ which is given by $S(\xi)(n) = \xi(n-1)$, $\xi(-1)=0$. Then $\mathcal{T}^A$ is generated by the elements $a \in l^2( \Ne, A)$ and $S  \in l^2( \Ne, A)$. We have the following exact sequence:

	\begin{equation}\label{eq:pimsner_short}
\begin{CD}
0 @>{}>>\mathcal{K}(l^2( \Ne, A))@>{\ph}>> \mathcal{T}^A @>{\ps}>>A\rtimes \Z @>{}>>0
\end{CD}
\end{equation}

\noindent by defining $\ps(a) = a$ and $\ps(S) = U$, where $U$ is the unitary in the crossed product  $A\rtimes \Z$ coming from the generator of $\Z$. It can be easily checked that $\ker(\ps) = A \otimes \mathcal{K}$. This is the so-called Pimsner--Voiculescu exact sequence which gives rise to the Pimsner--Voiculescu six term exact sequence. Now we define the map $e_2$ to be the exponential map in $\K$-theory for the above exact sequence. So we have 

{\small \[
\begin{CD}
\K_0 (\mathcal{K}(l^2( \Ne, A))) @>{}>>\K_0 (\mathcal{T}^A)@>{}>> \K_0 (A\rtimes \Z )  \\
@A{}AA & &  @VV{e_2}V            \\
 \K_1 (A\rtimes \Z ) @<<{}<  \K_1 (\mathcal{T}^A) @<<{}< \K_1 (\mathcal{K}(l^2( \Ne, A))).
\end{CD}
\]} 

\noindent Pimsner--Voiculescu also proved that $\mathcal{T}^A$ is $\KK$-equivalent to the algebra $A$. This gives

{\small \[
\begin{CD}
\K_0 (A) @>{\id -\gamma_{*}^{-1}}>>\K_0 (A)@>{i_*}>> \K_0 (A\rtimes \Z)   \\
@A{}AA & &  @VV{}V            \\
 \K_1 (A\rtimes \Z) @<<{i*}<  \K_1 (A) @<<{\id -\gamma_{*}^{-1}}< \K_1 (A),
\end{CD}
\]} 
where $i$ is the inclusion. 

The exact sequence \ref{eq:pimsner_short} also gives rise to the following exact sequence (tensoring with $M_2$)

\[
\begin{CD}
0 @>{}>>M_2( \mathcal{K}(l^2( \Ne, A)))@>{\ph}>> M_2(\mathcal{T}^A) @>{\ps}>>M_2(A\rtimes \Z) @>{}>>0.
\end{CD}
\]

We now describe the map $e_1$ using an exact sequence like the one above. 
Let the group $\mathbb{Z}_2 *  \mathbb{Z}_2$ be generated by $g$ and $s$, i.e. $g$ and $s$ generate the first and second copy of $\Z_2$ in $\mathbb{Z}_2 *  \mathbb{Z}_2$, respectively. Natsume obtained the exact sequence

\[
\begin{CD}
0 @>{}>>\mathcal{K}(l^2(P))@>{\eta}>> \mathcal{T}_p @>{\pi}>>C^* ( \mathbb{Z}_2 * \mathbb{Z}_2) @>{}>>0,
\end{CD}
\] with $P = P' \cup \{e\}$, where $P'$ is the the set of all non-empty words in $\mathbb{Z}_2* \mathbb{Z}_2$, which end in $g$ and $\mathcal{T}_p$ is generated by  $\mu(g)$ and $v(s)$, where $\mu(g)$ is the restriction of the left regular representation to $l^2(P)$ and  $v(s) =\ld(s)q(P)$, where $\ld(s)$ is the restriction of the left regular representation to $l^2(P') $  and $q(P)$ is the projection onto the subspace generated by the inclusion 
$l^2(P') \S l^2(P)$.  When all these are defined, there is a map $\pi$ sending $\mu(g)$ to $\ld(g)$ and $v(s)$ to $\ld(s)$. Denoting $\ker(\pi)$ to be $\mathcal{I}$, it can be shown that $\mathcal{I}$ is isomorphic to $\mathcal{K}(l^2(P))$. More details may be found in the paper by Natsume \cite{Na}.  Let $e_1$ be the map from $\K_0(C^* ( \mathbb{Z}_2 * \mathbb{Z}_2))$ to $\K_1(\mathcal{K}(l^2(P))$ coming from the  six-term exact sequence corresponding to the above exact sequence in $\K$-theory. 

The above construction can be easily extended to the case of crossed product. Let $\mathbb{Z}_2 *  \mathbb{Z}_2$ act on a unital $A$ with the action $\alpha$ and $\beta$ on $A$ from $\langle s \rangle$ and $\langle g \rangle$, respectively.  We call the action of $\mathbb{Z}_2 *  \mathbb{Z}_2$ by ${\alpha,\beta}$ and denote the crossed product by $A \rtimes_{\alpha,\beta} \Z_2*\Z_2$. 
$\mathcal{T}^A_{p}$ is constructed as follows. We have the natural representation $\iota'$ of $A \rtimes_{\alpha,\beta} \Z_2*\Z_2$ on $l^2(\Z_2*\Z_2, A)$ which we restrict to the Hilbert module  $l^2(P, A)$ in the following sense: if we denote the restriction by $\iota_2$, $a \in A, g, s$ act by the operators  
\begin{small}
$$\iota_2(a)(\xi) (x) = (\alpha,\beta)_{x^{-1}}(a)\xi (x), \quad \iota_2(g) (\xi) (x) = \xi (gx), \quad \iota_2(s) (\xi) (x) = \xi (sx)$$
\end{small}$(\text{by setting } \xi (s) = 0)$. Then $\mathcal{T}^A_{p}$ is the C*-algebra generated by $\iota_2(a)(a \in A), \iota_2(g)$ and $\iota_2(s)$. Now we have the following exact sequence (see Lemma A.3 \cite{Na}):

\[
\begin{CD}
0 @>{}>> \mathcal{K}(l^2(P,A))@>{\eta}>> \mathcal{T}^A_{p} @>{\pi}>>A \rtimes_{\alpha,\beta} \Z_2*\Z_2 @>{}>>0
\end{CD}
\] by defining $\pi(\iota_2(a)) = \iota'(a)$, $\pi(\iota_2(g)) = \iota'(g)$ and $\pi(\iota_2(s)) = \iota'(s)$.  Note that  for the case $A = \C, \iota_2(g) = \mu(g)$(as a generator of $\mathcal{T}_p$), $\iota_2(s) =v(s), \iota'(g) = \mu(g),  \iota'(s) = \mu(s)$. Denote the exponential map (of $\K$-theory) of the above exact sequence still by $e_1$.   The above exact sequence gives rise to Natsume's exact sequence in $\K$-theory. So we have
{\small \[
\begin{CD}
\K_0 (\mathcal{K}(l^2(P,A))) @>{}>>\K_0 (\mathcal{T}^A_{p})@>{}>> \K_0 (A \rtimes_{\alpha,\beta} \Z_2*\Z_2)  \\
@A{}AA & &  @VV{e_1}V            \\
 \K_1 (A \rtimes_{\alpha,\beta} \Z_2*\Z_2 ) @<<{}<  \K_1 (\mathcal{T}^A_{p})) @<<{}< \K_1 (\mathcal{K}(l^2(P,A))).
\end{CD}
\]} 
\noindent $\mathcal{T}^A_{p}$ can be shown to be $\KK$-equivalent to $(A \rtimes_ \alpha \Z_2) \oplus (A \rtimes_ \beta \Z_2)$ (see \cite{Na}, also \cite{PimsVoic82}).

Let $S=\mathbb{Z}\rtimes \mathbb{Z}_2$ with generators $a$ and $b$ i.e $a$ generates $\Z$, $b$ generates $\Z_2$, and $\Z_2$ acts on $\Z$ by flip. Now $S$ is isomorphic to $\mathbb{Z}_2 *  \mathbb{Z}_2$, where the later group is generated by $g$ and $s$ and the isomorphism identifies $a$ and $b$ with $sg$ and $g$, respectively. Now $l^2(P,A)$ could be identified with $l^2(P_1,A) \oplus l^2(P_2,A)$, where $$P_2 = \{g, gsg, gsgsg, gsgsgsg, \ldots\}, \quad P_1 =\{e, sg, sgsg, sgsgsg, \ldots\}.$$ Counting the number of $sg$'s, $P_1$ and $P_2$ have natural identifications with  $\Ne$. Under this identification $l^2(P,A)$ becomes $l^2(\Ne,A) \oplus l^2(\Ne,A)$, $\iota_2 (s)$ becomes $\left( \begin{array}{ccc}
0 & S\\
S^*& 0\\
  \end{array} \right)$, $\iota_2(g)$ becomes  $\left( \begin{array}{ccc}
0 & 1\\
1& 0\\
  \end{array} \right)$ and $\iota_2(a)$ becomes $\left( \begin{array}{ccc}
\gamma^{-1}(a) & 0\\
0& \gamma^{-1}\beta(a)\\
  \end{array} \right),$ where $\gamma$ is the action $\alpha \circ \beta$ on $A$. 
 Now under the above identification we have\begin{align*}
   \mathcal{K}(l^2(P,A)) &=  \mathcal{K}(l^2(P_1,A) \oplus l^2(P_2,A))
\\
 &=  \mathcal{K}(l^2(\Ne,A) \oplus l^2(\Ne,A))
\\  &= M_2(\mathcal{K}((l^2(\Ne,A))).
\end{align*} 
Also note that, from the generators and relations we have 
$$A \rtimes_{\alpha,\beta} \Z_2*\Z_2 = (A \rtimes_\gamma \Z) \rtimes \mathbb{Z}_2,$$ where, in the right hand side, the action of $\mathbb{Z}_2 = \langle g \rangle$ on $A$ is the given action $\beta$ and on $\Z$ is flip.

Summarising we have the inclusions $P^A_{\mathcal{K}}$, $P^A_{\mathcal{T}}$, $P^A_{\T}$ as follows:
\

$P^A_{\mathcal{K}} : \mathcal{K}(l^2(P,A)) \rightarrow M_2(\mathcal{K}((l^2(\Ne,A)))$, 

\

$P^A_{\mathcal{T}} : \mathcal{T}^A_{p} \rightarrow M_2(\mathcal{T}^A)$,  \hspace{.5cm} ($\mathcal{T}^A$ denotes the Toeplitz algebra of the $\Z$ action $\gamma$ on $A$)

\

$P^A_{\T} :   A \rtimes_{\alpha,\beta} \Z_2*\Z_2 = (A \rtimes_\gamma \Z) \rtimes \mathbb{Z}_2 \rightarrow M_2(A \rtimes_\gamma \Z), $\ 
which are given by

\begin{tiny}
$P^A_{\mathcal{T}} (\iota_2 (s)) = \left( \begin{array}{ccc}
0 & S\\
S^*& 0\\
  \end{array} \right), \quad P^A_{\mathcal{T}} (\iota_2(g)) = \left( \begin{array}{ccc}
0 & 1\\
1& 0\\
  \end{array} \right), \quad P^A_{\mathcal{T}} (\iota_2(a)) = \left( \begin{array}{ccc}
a & 0\\
0& \beta(a)\\
  \end{array} \right),$
  \end{tiny} $P^A_{\mathcal{K}}$ is the identification map, $P^A_{\T}$ is the natural map $p$ which was  discussed before i.e

\begin{tiny}
$P^A_{\mathbb{T}} (U) = \left( \begin{array}{ccc}
U & 0\\
0 & U^*\\
  \end{array} \right), \quad P^A_{\mathbb{T}} (W) = \left( \begin{array}{ccc}
0 & 1\\
1& 0\\
  \end{array} \right), \quad P^A_{\mathbb{T}} (\iota_2(a)) = \left( \begin{array}{ccc}
a & 0\\
0& \beta(a)\\
  \end{array} \right),$
  \end{tiny}
  where $U$ and $W$ are the unitaries corresponding to the generators of $\Z$ and $\Z_2$ inside $(A \rtimes_\gamma \Z) \rtimes \mathbb{Z}_2,$ respectively.
  
  \noindent So by construction we have the following commutative diagram:
\begin{small}
  	\[
\begin{CD}
0 @>{}>> M_2(\mathcal{K}(l^2(\Ne),A))@>{\ph}>> M_2(\mathcal{T}^A) @>{\ps}>>M_2(A\rtimes_\gamma \Z) @>{}>>0
\\
&& @AA{P^{A}_{\mathcal{K}}}A @AA{P^{A}_{\mathcal{T}}}A  @AA{P^{A}_{\mathbb{T}}}A            \\
0 @>{}>>\mathcal{K}(l^2(P),A)@>{\eta}>> \mathcal{T}^A_p @>{\pi}>> A \rtimes_{\alpha,\beta} \Z_2*\Z_2 @>{}>>0
\end{CD}
\]
 \end{small}

\begin{thm}\label{main1}

For unital $A$, the connecting maps of the above sequences (Pimsner--Voiculescu and Natsume) commute in the following sense:
\[
\xymatrixrowsep{5pc}
\xymatrixcolsep{5pc}
\xymatrix
{
\K_0(A \rtimes_{\alpha,\beta} \Z_2*\Z_2) \ar[r]^{e_1} \ar[rd]^{p_*} &
\K_1 (A)  \\
&\K_0(A\rtimes_\gamma \Z)),\ar[u]^{e_2}}
\]
where $p_*$ is the map induced by the natural map $p:A \rtimes_{\alpha,\beta} \Z_2*\Z_2 \cong   (A\rtimes_\gamma \Z)\rtimes\Z_2 \rightarrow M_2(A\rtimes_\gamma \Z).$ 
\end{thm}

\begin{proof}
	
	The result follows from the  commutative diagram 
	\begin{small}
	    	\[
\begin{CD}
0 @>{}>> M_2(\mathcal{K}(l^2(\Ne),A))@>{\ph}>> M_2(\mathcal{T}^A) @>{\ps}>>M_2(A\rtimes_\gamma \Z) @>{}>>0
\\
&& @AA{P^{A}_{\mathcal{K}}}A @AA{P^{A}_{\mathcal{T}}}A  @AA{P^{A}_{\mathbb{T}}}A            \\
0 @>{}>>\mathcal{K}(l^2(P),A)@>{\eta}>> \mathcal{T}^A_p @>{\pi}>> A \rtimes_{\alpha,\beta} \Z_2*\Z_2 @>{}>>0
\end{CD}
\]\end{small}
  and the naturality of connecting maps. 
\end{proof}

\subsection{Main computations with $\K$-theory}\label{Sec:nt}

Let us consider the $3 \times 3$  skew-symmetric matrix: 

$$\theta
  := \left( \begin{array}{cccc}
0 & \theta_{12} & \theta_{13}\\
  -\theta_{12}  & 0 & \theta_{23}\\
    -\theta_{23}  & -\theta_{23} & 0\\
  \end{array}\right), $$
  where $\theta_{12}, \theta_{13},\theta_{23}$ are irrational numbers and rationally independent. We call such a $\theta$ totally irrational. Let $A_{\te}$ be the corresponding 3-dimensional noncommutative torus generated by $u_1$, $u_2$ and $u_3$.  Let $\mathbb{Z}_2 = \langle g\rangle $ act on $A_{\te}$  by flip. We denote by $A_{\theta_{12}},$ the two dimensional noncommutative tori, which is generated by the matrix $$ \left( \begin{array}{ccc}
0 & \theta_{12} \\
  -\theta_{12}  & 0 \\
   \end{array}\right).$$ $\mathbb{Z}_2 = \langle g\rangle $ also acts on $A_{\theta_{12}}$ by flip. Let us also define an action of an another copy of $\mathbb{Z}_2 = \langle s\rangle $ on $A_{\theta_{12}}$  given by $u_1 \rightarrow e(-\theta_{13})u_1^{-1}$ and $u_2 \rightarrow e(-\theta_{23})u_2^{-1}$. So we have corresponding crossed products which we denote by $A_{\theta_{12}}\rtimes _g \Z_2$ and $A_{\theta_{12}}\rtimes _s \Z_2$, respectively. So we have, in fact, $\mathbb{Z}_2*\Z_2 = \langle g,s\rangle $ acting on $A_{\theta_{12}}$. The corresponding crossed product we denote by $A_{\theta_{12}}\rtimes_\phi (\mathbb{Z}_2 * \mathbb{Z}_2 ).$  As we have seen in the beginning of this section that $A_\theta$ can be written as crossed product $A_{\theta_{12}} \rtimes_\gamma \Z$ by the action $u_1 \rightarrow e(\theta_{13})u_1$ and $u_2 \rightarrow e(\theta_{23})u_2,$ 
    we have  the isomorphism $ A_\te\rtimes  \mathbb{Z}_2 \cong (A_{\theta_{12}} \rtimes_\gamma \Z) \rtimes \Z_2 \cong  A_{\theta_{12}}\rtimes_\phi (\mathbb{Z}_2 * \mathbb{Z}_2 ) .$ Note that, in the isomorphism $ A_{\theta_{12}}\rtimes_\phi (\mathbb{Z}_2 * \mathbb{Z}_2 ) \cong A_\te\rtimes  \mathbb{Z}_2 ,$ $\lambda(sg)$ is identified with $u_3$ and $\lambda(g)$ is identified with $w$, where $w$ is the unitary in $A_\te\rtimes  \mathbb{Z}_2$ coming from the unitary in $\mathbb{Z}_2$.

In this case, $A_{\theta_{12}}\rtimes_\phi (\mathbb{Z}_2 * \mathbb{Z}_2 )$, Natsume's exact sequence looks like 

{\tiny \[
\begin{CD}
\K_0 (A_{\theta_{12}}) @>{i_{1*}-i_{2*}}>>\K_0 (A_{\theta_{12}}\rtimes _s \Z_2)\oplus \K_0 (A_{\theta_{12}}\rtimes_g \Z_2)@>{j_{1*}+j_{2*}}>> \K_0 (A_{\theta_{12}}\rtimes_\phi \Z_2*\Z_2)   \\
@A{}AA & &  @VV{}V            \\
 \K_1 (A_{\theta_{12}}\rtimes_\phi \Z_2*\Z_2) @<<{j_{1*}+j_{2*}}<  \K_1 (A_{\theta_{12}}\rtimes_s \Z_2)\oplus \K_1 (A_{\theta_{12}}\rtimes_g \Z_2) @<<{i_{1*}-i_{2*}}< \K_1 (A_{\theta_{12}}),
\end{CD}
\]} where $i_1, i_2, j_1,j_2$ are inclusions.

\subsection{A basis of $\K_0(A_\theta)$}\label{subsec:basis_of_Atheta}

Let us first write down a basis $\K_0(A_\theta)$. We have  
 $$\theta = \left( \begin{array}{ccc}
0 & \theta_{12} & \theta_{13}\\ 
 -\theta_{12} & 0  & \theta_{23}\\
-\theta_{13} & -\theta_{23} & 0\\
\end{array} \right) .
$$ As before, we also assume that $\theta$ is totally irrational. Following Elliott (\cite[Theorem 3.1]{elliott80}), 
since $\theta$ is totally irrational, the trace (which we will denote by $\Tr$ always) of $A_\theta$ is faithful and the range of the trace is the subgroup of $\R$ generated by the numbers $1, \theta_{12}, \theta_{13}, \theta_{23}$. Also a similar result holds for $A_{\theta_{12}}$ i.e the trace of $A_{\theta_{12}}$ is faithful and the range is the subgroup of $\R$ generated by the numbers $1, \theta_{12}$.    We consider the projective $A_\theta$-module $S^1:=\overline{ \mathscr{S}(\R \times \Z)_{12}}$ constructed as in Theorem \ref{thm:rieffel_proj_mod} for $M = \R \times \Z$. The trace of this module is $\theta_{12}$ (see Remark \ref{rmk:trace_of_heisenberg}).

Now consider the following matrices  $$
 \theta_1 = \left( \begin{array}{ccc}
0 & \theta_{23} & -\theta_{12}\\
 -\theta_{23} & 0  & -\theta_{13}\\
\theta_{12} & \theta_{13} & 0\\
\end{array} \right), 
 $$ $$
  \theta_2 = \left( \begin{array}{ccc}
0 & \theta_{13} & \theta_{12}\\
 -\theta_{13} & 0  & -\theta_{23}\\
-\theta_{12} & \theta_{23} & 0\\
\end{array} \right) .
$$

Note that $A_{\theta}$, $A_{\theta_1}$ and $A_{\theta_2}$ re\-present the same noncommutative torus. Let $S^2:=\overline{ \mathscr{S}(\R \times \Z)_{13}}$ and $S^3:=\overline{ \mathscr{S}(\R \times \Z)_{23}}$ be the projective modules over $A_{\theta_2} $ and $A_{\theta_1}$, respectively, as discussed above. Now, similarly to the previous case, we see that $\Tr(S^2)= \theta_{13}$ and $\Tr(S^3) = \theta_{23}$. Since $\Tr$ is faithful for $\theta$ (as it is totally irrational), we conclude that  $[S^1],[S^2],[S^3]$ along with the trivial element generate $\K_0(A_\theta)$.  

Let us now recall the Pimsner--Voiculescu exact sequence for $A_\theta = A_{\theta_{12}}\rtimes_\gamma\Z$

{\small \[
\begin{CD}
\K_0 (A_{\theta_{12}}) @>{0}>>\K_0 (A_{\theta_{12}})@>{i_*}>> \K_0 (A_{\theta_{12}}\rtimes \Z)   \\
@A{}AA & &  @VV{e_2}V            \\
 \K_1 (A_{\theta_{12}}\rtimes \Z) @<<{i*}<  \K_1 (A_{\theta_{12}}) @<<{0}< \K_1 (A_{\theta_{12}}).
\end{CD}
\]} 

\noindent In the exact sequence the upper left and the lower right map is 0 because the action $\gamma$ of $\Z$ on $A_{\theta_{12}}$ is homotopic to the trivial action (\cite[Lemma 1.5]{Ph})). 
This gives rise to the short exact sequence:

\[
\begin{CD}
0 @>{}>>\K_0 (A_{\theta_{12}})@>{i_*}>> \K_0 (A_{\theta_{12}}\rtimes_\gamma \Z) @>{e_2}>>\K_1 (A_{\theta_{12}}) @>{}>>0.
\end{CD}
\] 

Let us consider the projective $A_{\theta_{12}}$-module $S:=\overline{ \mathscr{S}(\R)_{12}}$ constructed as in Theorem \ref{thm:rieffel_proj_mod} for $M = \R$. The trace of this module is $\theta_{12}$ (see Remark \ref{rmk:trace_of_heisenberg}). Now, since the image of the faithful trace of $A_{\theta_{12}}$ is generated by the numbers 1 and $\theta_{12},$  a basis of $\K_0 (A_{\theta_{12}})$ is given by $\{[1],[S]\}$. Also in the above short exact sequence we have $i_*([S]) = [S^1],$ since they have the same trace in $A_\theta = A_{\theta_{12}}\rtimes \Z$ and $i_*([1]) = [1].$  Since  $[1], [S^1], [S^2], [S^3]$ generate $\K_0$ of $A_\te$, $[S^2], [S^3]$ map to generators of $\K_1 (A_{\theta_{12}})$ by the map $e_2$ in the above short exact sequence.

\subsection{A basis of $\K_0(A_{\theta_{12}}\rtimes_g\Z_2)$}
Note that $A_{\theta_{12}}\rtimes_g\Z_2$ is generated by the unitaries $u_1$, $u_2$ and $w$ such that $w^2 = 1, u_2u_1 = e(\theta_{12})u_1u_2, wuw = u^*, wvw = v^*.$ A basis of $\K_0(A_{\theta_{12}}\rtimes_g\Z_2)$ was given by Walters in \cite[Lemma 2.3]{walters:k-theoryflip} (also see \cite{ELPW}). Let $S^{\theta_{12}}_g$ be the  special $\K$-theory element which we call ``flip invariant Rieffel projection" inside  $A_{\theta_{12}}\rtimes _g\Z_2$ as constructed by Walters (\cite{walters:k-theoryflip})(the element was denoted by $E_{00}(\theta)$ and $F_{00}(\theta)$ in \cite[Page 593]{walters:k-theoryflip}) in the same way as the classical Rieffel projection  $S^{\theta_{12}} \in A_{\theta_{12}}$. Then a basis of $\K_0(A_{\theta_{12}}\rtimes_g\Z_2)$ is given by the $\K$-theory classes of the following elements:\\
$1$, \\
 $P_1= \frac{1}{2} (1+w)$, \\
 $P_2= \frac{1}{2} (1-u_1w)$, \\
 $P_3= \frac{1}{2} (1-u_2w)$, \\
 $P_4= \frac{1}{2} (1-e( \frac{1}{2} {\theta_{12}})u_1u_2w)$ and \\ 
 $S^{\theta_{12}}_g.$ Also it is well known that $\K_1(A_{\theta_{12}}\rtimes_g\Z_2)$ is trivial (see for example \cite{walters:k-theoryflip}, \cite{ELPW}).
 \subsection{A basis of $\K_0(A_{\theta_{12}}\rtimes_s\Z_2)$}$A_{\theta_{12}}\rtimes_s\Z_2$ is generated by the unitaries $u_1$, $u_2$ and $w'=u_3w$ and  we have the relations $w'^2 = 1, u_2u_1 = e(\theta_{12})u_1u_2, w'u_1w' = e(-\theta_{13})u_1^{-1}, w'u_2w' = e(-\theta_{23})u_2^{-1}.$ Upon setting $\tilde{u_1} = e(\frac{1}{2}\theta_{13})u_1$ and $\tilde{u_2} = e(\frac{1}{2}\theta_{23})u_2$, we have that $$\tilde{u}_2\tilde{u}_1 = e(\theta_{12})\tilde{u}_1\tilde{u}_2,\quad  w'\tilde{u}w' = \tilde{u}^*,\quad w'\tilde{v}w' = \tilde{v}^*.$$ So $A_{\theta_{12}}\rtimes_s\Z_2$ is isomorphic to $A_{\theta_{12}}\rtimes_g\Z_2$ and a basis of $A_{\theta_{12}}\rtimes_s\Z_2$ is given by the $\K$-theory classes of the following elements:\\
 $1$, \\
 $P_5= \frac{1}{2} (1+u_3w)$, \\
 $P_6= \frac{1}{2} (1-u_1u_3w)$, \\
 $P_7= \frac{1}{2} (1-u_2u_3w)$, \\
 $P_8= \frac{1}{2} (1-e( \frac{1}{2} ({\theta_{12}+\theta_{13}+\theta_{23}}))u_1u_2u_3w)$ and \\ 
 $S^{\theta_{12}}_s,$ where $S^{\theta_{12}}_s$ is the ``flip invariant Rieffel projection" inside  \mbox{$A_{\theta_{12}}\rtimes_s\Z_2.$}

 As shown in the beginning of the section, the projective modules $S^2$,  $S^3$ can be extended to projective modules over $A_{\theta} \rtimes \Z_2\cong A_{\theta_{12}}\rtimes_\phi \Z_2*\Z_2$ and by abuse of notations, we still denote the extended modules by $S^2$,  $S^3$, as well.

\begin{cor}\label{3toriflipped}
Let $\theta$ be a totally irrational $3 \times 3$ skew-symmetric matrix. Then $\K_0(A_{\te}\rtimes \mathbb{Z}_2)$ is isomorphic to $\mathbb{Z}^{12}$, which is generated by the $\K$-theory classes of the elements: \\
1,\\
$P_1= \frac{1}{2} (1+w)$,\\
$P_2= \frac{1}{2} (1-u_1w)$,\\
$P_3= \frac{1}{2} (1-u_2w)$,\\
$P_4= \frac{1}{2} (1-e( \frac{1}{2} {\theta_{12}})u_1u_2w)$,\\
$P_5= \frac{1}{2} (1+u_3w)$,\\
$P_6= \frac{1}{2} (1-e( \frac{1}{2} {\theta_{13}})u_1u_3w)$,\\
$P_7= \frac{1}{2} (1-e( \frac{1}{2} {\theta_{23}})u_2u_3w)$,\\
$j_{2*}(S^{\theta_{12}}_g)$, $j_{1*}(S^{\theta_{12}}_s)$, $S^2$, $S^3$. 
\end{cor}

\begin{proof}

We have Natsume's exact sequence for the C*-algebra $A_{\te}\rtimes \mathbb{Z}_2 \cong A_{\theta_{12}}\rtimes_\phi \Z_2*\Z_2$ :
{\tiny \[
\begin{CD}
\K_0 (A_{\theta_{12}}) @>{i_{1*}-i_{2*}}>>\K_0 (A_{\theta_{12}}\rtimes _s \Z_2)\oplus \K_0 (A_{\theta_{12}}\rtimes_g \Z_2)@>{j_{1*}+j_{2*}}>> \K_0 (A_{\theta_{12}}\rtimes_\phi \Z_2*\Z_2)   \\
@A{}AA & &  @VV{}V            \\
 \K_1 (A_{\theta_{12}}\rtimes_\phi \Z_2*\Z_2) @<<{j_{1*}+j_{2*}}<  \K_1 (A_{\theta_{12}}\rtimes_s \Z_2)\oplus \K_1 (A_{\theta_{12}}\rtimes_g \Z_2) @<<{i_{1*}-i_{2*}}< \K_1 (A_{\theta_{12}}).
\end{CD}
\]}

It is well known that $\K_0 (A_{\theta_{12}})$ is generated by $1$ and the class of the Bott element $[S^{\theta_{12}}].$ Now for $0 <  \theta_{12} < \frac{1}{2}$, $$i_{2*}[S^{\theta_{12}}] = 2[S^{\theta_{12}}_g] + ([P_2] + [P_4]) - ([P_1] + [P_3]).$$ Indeed, these two elements have the same vector trace \cite[Corollary 5.6, and Page 597]{walters:k-theoryflip} (the first named author is indebted to Siegfried Echterhoff for showing the argument.) Similarly, $$i_{1*}[S^{\theta_{12}}] = 2[S^{\theta_{12}}_s] + ([P_6] + [P_8]) - ([P_5] + [P_7]).$$ Also for $\frac{1}{2} <  \theta_{12} < 1$, the expression of $i_{1*}[S^{\theta_{12}}]$ is essentially the same (with some sign modification) and hence gives a similar result.  Hence upper left $i_{1*}-i_{2*}$ is injective. Also $\K_1 (A_{\theta_{12}}\rtimes_s \Z_2)\oplus \K_1 (A_{\theta_{12}}\rtimes_g \Z_2)$ is zero. So $\K_1(A_{\theta_{12}}\rtimes_\phi \Z_2*\Z_2)$ is also zero. Now we are left with

{\tiny \[
\begin{CD}
\K_0 (A_{\theta_{12}}) @>{i_{1*}-i_{2*}}>>\K_0 (A_{\theta_{12}}\rtimes _s \Z_2)\oplus \K_0 (A_{\theta_{12}}\rtimes_g \Z_2)@>{j_{1*}+j_{2*}}>> \K_0 (A_{\theta_{12}}\rtimes_\phi \Z_2*\Z_2)   \\
@A{}AA & &  @VV{}V            \\
 0 @<<{0}<  0. @<<{0}< \K_1 (A_{\theta_{12}}).
\end{CD}
\]}

Now the elements $S^2$ and $S^3$ are lifted from $\K_0(A_{\theta_{12}} \rtimes \Z)$ to $\K_0(A_{\theta_{12}} \rtimes_\phi \Z_2*\Z_2)$ via the map $p_*$ (since we can think $p_*$ as inverse of Green--Julg map). Since $S^2$ and $S^3$ map to generators of $\K_1 (A_{\theta_{12}})$ (in the Pimsner--Voiculescu sequence via the map $e_2$), by Theorem \ref{main1} the claim follows. 
\end{proof}
Since the $\K$-theory of $A_\theta \rtimes \Z_2$ does not dependent on the parameter $\theta$ (see \cite[Corollary 1.13]{ELPW}), we obtain 
		
		\begin{cor}
		If $\theta$ is a $3 \times	3$ real skew-symmetric matrix, then $\K_0(A_\theta \rtimes \Z_2) = \mathbb{Z}^{12}$ and $\K_1(A_\theta \rtimes \Z_2) = 0$.
		\end{cor}
		
		\begin{rmk}
		Also for a general $\theta$, a continuous field argument (\cite[Remark 2.3]{ELPW}, \cite{chak01}) could be used to find the $\K$-theory generators, which we shall omit.

		\end{rmk}

\begin{rmk}\label{FWremark}
	The authors have the unpleasant obligation to warn the reader, that the corresponding result for $\K$-theory of flipped crossed product in \cite{FW} has a gap: it was claimed that (in the notations of the present paper) $i_{1*} - i_{2*} $ sends generators to  generators. Fortunately, the arguments in the last corollary fill the gap for the three dimensional case which do not affect the final result of  \cite{FW} at least for the three dimensional case.
	\end{rmk}

  \begin{qst}
	Our results suggest the following questions:
	\begin{itemize}
		\item Can one extend the Heisenberg modules of the form $\mathscr{S}(\R^p \times \Z^q)$ over $A_\theta$ to ones over general crossed products  $A_{\theta}^\infty \rtimes F$, where the finite cyclic group $F$ acts on $A_{\theta}^\infty$ as in Section 4?
		\item 
	What are the generators of  $\K_0(A_{\theta} \rtimes F)$, where $F$ acts on $A_{\theta}$ as in Section 4?
	  \item Heisenberg modules are also linked with results in signal analysis, concretely with Gabor frames, as shown in \cite{lu09}. Thus one might wonder about the consequences of our results for the theory of Gabor frames. 
	  \end{itemize}
  \end{qst}



\begin{thebibliography}{333}
\bibitem{BEEK91}
O. Bratteli, G.A. Elliott, D.E. Evans, and A Kishimoto.
\newblock Noncommutative spheres. {I}.
\newblock {\em Internat. J. Math.}, 2(2):139--166, 1991.

\bibitem{buwa07-1}
J.~{B}uck and S.~{W}alters.
\newblock {C}onnes-{C}hern characters of hexic and cubic modules.
\newblock {\em {J}. {O}perator {T}heory}, 57(1):35--65, 2007.


\bibitem{connes:bigbook}
A.~{C}onnes.
\newblock {\em {N}oncommutative {g}eometry}.
\newblock {A}cademic {P}ress {I}nc., {S}an {D}iego, {C}{A}, 1994. 

\bibitem{chak01}
S.~{C}hakraborty.
\newblock {\em Some remarks on $K_0$ of noncommutative tori}.
\newblock Preprint.  	arXiv:1704.03349 


\bibitem{ELPW}
S.~{E}chterhoff, W.~{L}{\"u}ck, N.~{P}hillips, and S.~{W}alters. 
\newblock {T}he structure of crossed products of irrational rotation algebras by finite subgroups of {${\rm SL}_2(\mathbb{Z})$}.
\newblock {\em J. Reine Angew. Math.}, 639:173--221, 2010.

\bibitem{elliott80}
G.~A. Elliott.
\newblock On the {$K$}-theory of the {$C^{\ast} $}-algebra generated by a
  projective representation of a torsion-free discrete abelian group.
\newblock In {\em Operator algebras and group representations, {V}ol. {I}
  ({N}eptun, 1980)}, volume~17 of {\em Monogr. Stud. Math.}, pages 157--184.
  Pitman, Boston, MA, 1984.

 \bibitem{FW} C.\ Farsi, N.\ Watling.
\newblock {S}ymmetrized non-commutative tori.
\newblock {\em {M}ath. {A}nn.} 296 (1993), 739-741.

\bibitem{Go}
M.~de {G}osson.
\newblock {\em {S}ymplectic {M}ethods in {H}armonic {A}nalysis and in
{M}athematical {P}hysics}, volume~7 of {\em {P}seudo-{D}ifferential
{O}perators. {T}heory and {A}pplications}.
\newblock {B}irkh{\"a}user/{S}pringer {B}asel {A}{G}, {B}asel, 2011. 


\bibitem{jh}
J.~ A. {J}eong, and J.~{H}. Lee.
\newblock  {F}inite groups acting on higher dimensional noncommutative tori.
 \newblock {\em {J}. {F}unct. {A}nal.}, 268(2):473--499, 2015.

\bibitem{kumjian90}
A.~Kumjian.
\newblock On the {$K$}-theory of the symmetrized noncommutative torus.
\newblock {\em C. R. Math. Rep. Acad. Sci. Canada}, 12(2-3):87--89, 1990.
\bibitem{Li03}
H.~{L}i.
\newblock {S}trong Morita equivalence of higher-dimensional noncommutative tori. 
\newblock{\em J. Reine Angew. Math.}, 576:167--180, 2003.    

\bibitem{lu09}
F.~{L}uef.
\newblock {P}rojective modules over non-commutative tori are multi-window
{G}abor frames for modulation spaces.
\newblock {\em J. Funct. Anal.}, 257(6):1921--1946, 2009. 

\bibitem{Na}
T.  ~{N}atsume.
\newblock  {O}n $K_*(C^{*}(SL_2(\Z))$. {A}ppendix to "$\K$-theory for certain group $C^*$-algebras".
\newblock{\em {J}. Operator {T}heory}, 13(1):103–118, 1985. 
 
\bibitem{Ph}
N.~{P}hillips.
\newblock {E}very simple higher dimensional noncommutative torus is an {A}{T}
  algebra.
\newblock {\em ar{X}iv preprint math/0609783}, 2006.

 \bibitem{PimsVoic80} M.\ Pimsner, D.\ Voiculescu.
\newblock {E}xact sequences for $\K$-groups and $\mathrm{Ext}$-groups of certain cross-products $\Cst$-algebras.
\newblock{\em {J}. Operator {T}heory}, 4:93-118, 1980.

\bibitem{PimsVoic82} M.\ Pimsner, D.\ Voiculescu. 
\newblock $\K$-groups of reduced crossed product by free groups. 
\newblock{\em {J}. Operator {T}heory}, 8:131-156, 1982.

\bibitem{ri88}
M.~A. {R}ieffel.
\newblock {P}rojective modules over higher-dimensional noncommutative tori.
\newblock {\em {C}anad. {J}. {M}ath.}, 40(2):257--338, 1988.

\bibitem{risc99}
M.~{R}ieffel and A.~{S}chwarz.
\newblock {M}orita equivalence of multidimensional noncommutative tori.
\newblock {\em {I}nt. {J}. {M}ath.}, 10(02):289--299, 1999. 

\bibitem{walters:k-theoryflip}
S.~{W}alters., Projective modules over the non-commutative sphere,
{\em J. London Math. Soc.} (2) 51 (1995), no. 3, 589-602.

\bibitem{Wal00}
S.~{W}alters.
\newblock Chern characters of {F}ourier modules.
\newblock {\em Canad. J. Math.}, 52(3):633--672, 2000.

\bibitem{wa01-2}
S.~{W}alters.
\newblock $\K$-theory of non-commutative spheres arising from the {F}ourier
automorphism.
\newblock {\em Canad. J. Math.}, 53(3):631--672, 2001. 

\bibitem{wa04}
S.~{W}alters.
\newblock {P}eriodic Integral Transforms and $C^*$-algebras
 \newblock	{\em C. R. Math. Rep. Acad. Sci. Canada}, 26(2):55-61, 2004. 

 \bibitem{we64}
A.~{W}eil.
\newblock {S}ur certains groupes d'op{\'e}rateurs unitaires.
\newblock {\em Acta Math.}, 111:143--211, 1964. 

\bibitem{zeller-meier68}
G.~Zeller-Meier.
\newblock Produits crois\'{e}s d'une {$C^{\ast} $}-alg\`ebre par un groupe
  d'automorphismes.
\newblock {\em J. Math. Pures Appl. (9)}, 47:101--239, 1968.

\end{thebibliography}
\end{document}